\documentclass[preprint,9pt]{elsarticle}
\usepackage{amssymb}
\usepackage{amsmath}
\usepackage{amsthm}
\usepackage{enumerate}

\newtheorem{thrm}{Theorem}[section]
\newtheorem{lem}[thrm]{Lemma}

\newtheorem{exam}[thrm]{Example}
\newtheorem{cor}[thrm]{Corollary}

\theoremstyle{definition}

\journal{...}
\input{tcilatex}

\begin{document}

\begin{frontmatter}


\cortext[cor1]{Corresponding author (+903562521616-3087)}

\title{On the representation of $k$ sequences of generalized order-$k$
numbers}

%\author{}
%\address{}

\author[rvt]{Kenan Kaygisiz\corref{cor1}}
\ead{kenan.kaygisiz@gop.edu.tr} {\author[rvt]{Adem Sahin }}
\ead{adem.sahin@gop.edu.tr}
\address[rvt]{Department of Mathematics, Faculty of Arts and Sciences,
Gaziosmanpa\c{s}a University, 60250 Tokat, Turkey}

\begin{abstract}
In this paper, we define $k$-generalized order-$k$ numbers and we
obtain a relation between $i$-th sequences and $k$-th sequences of
$k$-generalized order-$k$ numbers. We give some determinantal and
permanental representations of $k$-generalized order-$k$ numbers by
using various matrices. Using the relation between $i$-th sequences
and $k$-th sequences of $k$-generalized order-$k$ numbers we give
some determinantal and permanental representations of $i$-th
sequences of generalized order-$k$ numbers. In addition, we obtain
Binet's formula for generalized order-$k$ Pell numbers by using our
representations.
\end{abstract}
\begin{keyword}
Order-$k$ Fibonacci numbers, $k$ sequences of the generalized
order-$k$ Fibonacci numbers, $k$ sequences of the generalized
order-$k$ Pell numbers, Hessenberg Matrix.
\end{keyword}

\end{frontmatter}


%%------------------------INTRODUCTION---------------------------

\section{Introduction}

%---------------------------------------------------------------

Fibonacci numbers, Pell numbers and their generalizations have been
studying for a long time. One of these generalizations was given by
Miles in 1960.

Miles [6] defined generalized order-$k$ Fibonacci numbers(GO$k$F) as,%
\begin{equation}
f_{k,n}=\sum\limits_{j=1}^{k}f_{k,n-j}\
\end{equation}%
for $n>k\geq 2$, with boundary conditions:
$f_{k,1}=f_{k,2}=f_{k,3}=\cdots =f_{k,k-2}=0$ and
$f_{k,k-1}=f_{k,k}=1.$\newline

Er [3] defined $k$ sequences of generalized order-$k$ Fibonacci numbers ($k$%
SO$k$F) as; for $n>0,$ $1\leq i\leq k$%
\begin{equation}
f_{k,n}^{\text{ }i}=\sum\limits_{j=1}^{k}c_{j}f_{k,n-j}^{\text{ }i}\
\
\end{equation}%
with boundary conditions for $1-k\leq n\leq 0,$

\begin{equation*}
f_{k,n}^{\text{ }i}=\left\{
\begin{array}{l}
1\text{ \ \ \ \ \ if \ }i=1-n, \\
0\text{ \ \ \ \ \ otherwise,}%
\end{array}%
\right.
\end{equation*}%
where $c_{j}$ $(1\leq j\leq k)$ are constant coefficients, $f_{k,n}^{\text{ }%
i}$ is the $n$-th term of $i$-th sequence of order $k$ generalization. For $%
c_{j}=1$, $k$-th sequence of this generalization involves the Miles
generalization(1) for $i=k,$ i.e.%
\begin{equation}
f_{k,n}^{k}=f_{k,k+n-2}.
\end{equation}

\bigskip

Kili\c{c}[5] defined $k$ sequences of generalized order-$k$ Pell numbers ($k$%
SO$k$P) as; for $n>0,$ $1\leq i\leq k$%
\begin{equation}
p_{k,n}^{\text{ }i}=2p_{k,n-1}^{\text{ }i}+p_{k,n-2}^{\text{
}i}+\cdots +\ p_{k,n-k}^{\text{ }i}\
\end{equation}%
with initial conditions for $1-k\leq n\leq 0,$

\begin{equation*}
p_{k,n}^{\text{ }i}=\left\{
\begin{array}{l}
1\text{ \ \ \ \ \ if \ }i=1-n, \\
0\text{ \ \ \ \ \ otherwise,}%
\end{array}%
\right.
\end{equation*}%
where $p_{k,n}^{\text{ }i}$ is the $n$-th term of $i$-th sequence of order $%
k $ generalization.

\bigskip

\bigskip We give general form of sequences mentioned above depending on $%
\lambda \in
%TCIMACRO{\U{2124} }%
%BeginExpansion
\mathbb{Z}
%EndExpansion
^{+}$ and this sequences named as $k$ sequences of generalized
order-$k$
numbers ($k$SO$k$) as; for $n>0,$ $1\leq i\leq k$%
\begin{equation}
a_{k,n}^{\text{ }i}=\lambda a_{k,n-1}^{i}+a_{k,n-2}^{i}+\cdots +\
a_{k,n-k}^{i\text{ }}\
\end{equation}%
with initial conditions for $1-k\leq n\leq 0,$

\begin{equation*}
a_{k,n}^{i\text{ }}=\left\{
\begin{array}{l}
1\text{ \ \ \ \ \ if \ }i=1-n, \\
0\text{ \ \ \ \ \ otherwise,}%
\end{array}%
\right.
\end{equation*}%
where $a_{k,n}^{i\text{ }}$ is the $n$-th term of \ $i$-th sequences
of order-$k$ generalization.

\bigskip

Not that $f_{k,k+n-2},$ $f_{k,n}^{\text{ }i},$ $p_{k,n}^{\text{ }i}$ and $%
a_{k,n}^{\text{ }i}$ are GO$k$F, $k$SO$k$F, $k$SO$k$P and $k$SO$k$
respectively,

substituting $c_{j}=1$ in(2) and $\lambda =1$ in(5), for $1\leq
i\leq k$
we obtain%
\begin{equation*}
a_{k,n}^{\text{ }i}=f_{k,n}^{\text{ }i},
\end{equation*}%
\ \ substituting $\lambda =2$ in(5), for $1\leq i\leq k$ we obtain%
\begin{equation*}
a_{k,n}^{\text{ }i}=p_{k,n}^{\text{ }i}
\end{equation*}%
and substituting $\lambda =2$ in(5) we obtain
\begin{equation*}
a_{k,n}^{\text{ }k}=f_{k,k+n-2}.
\end{equation*}

\subsection{ Relation between $i$-th sequences and $k$-th sequences of
generalized order-$k$ numbers}

\begin{lem}
Let$\ a_{k,n}^{\text{ }i}$ be the $i$-th sequences of $k$SO$k$ for
$n>1-k$
and $1\leq i<k,$%
\begin{equation*}
a_{k,n}^{\text{ }i}=a_{k,n}^{\text{ }i+1}+\text{
}a_{k,n-k+i}^{\text{ }k}
\end{equation*}%
and since $a_{k,n-k+i}^{\text{ }k}=0$ for $1\leq n\leq k-i$ then $a_{k,n}^{%
\text{ }i}=a_{k,n}^{\text{ }i+1}.$
\end{lem}

\begin{proof}
\bigskip It's obvious that for $1\leq n\leq k-i,$\ $a_{k,n}^{\text{ }%
i}=a_{k,n}^{\text{ }i+1}$ from (5).

Assume for\ $n>k-i,$ $a_{k,n}^{\text{ }i}-a_{k,n}^{\text{
}i+1}=t_{n}$ and show $t_{n}=$ $a_{k,n-k+i}^{\text{ }k}.$

First we obtain initial conditions for $t_{n}$ by using initial
conditions
of \ $i$-th and $(i+1)$-th sequences of $k$SO$k$ simultaneously as follows;%
\begin{equation*}
\begin{tabular}{|c|c|c|c|}
\hline
$n\setminus $ & $a_{k,n}^{\text{ }i}$ & $a_{k,n}^{\text{ }i+1}$ & $%
t_{n}=a_{k,n}^{\text{ }i}-a_{k,n}^{\text{ }i+1}$ \\ \hline $1-k$ &
$0$ & $0$ & $0$ \\ \hline $2-k$ & $0$ & $0$ & $0$ \\ \hline $\vdots
$ & $\vdots $ & $\vdots $ & $\vdots $ \\ \hline $-i-1$ & $0$ & $0$ &
$0$ \\ \hline $-i$ & $0$ & $1$ & $-1$ \\ \hline $-i+1$ & $1$ & $0$ &
$1$ \\ \hline $-i+2$ & $0$ & $0$ & $0$ \\ \hline $\vdots $ & $\vdots
$ & $\vdots $ & $\vdots $ \\ \hline
$0$ & $0$ & $0$ & $0$%
\end{tabular}%
\end{equation*}

Since initial conditions of $t_{n}$ are equal to initial condition of $%
a_{k,n}^{\text{ }k}$ with index iteration and since $t_{-i+1}=a_{k,1-k}^{%
\text{ }i}$ then we have%
\begin{equation*}
t_{n}=a_{k,n-k+i}^{\text{ }k}.
\end{equation*}
\end{proof}

\begin{thrm}
\bigskip Let$\ a_{k,n}^{\text{ }i}$ be the $i$-th sequences of $k$SO$k,$
then for $n\geq 1$ and $1\leq i\leq k$%
\begin{equation*}
a_{k,n}^{\text{ }i}=a_{k,n}^{\text{ }k}+a_{k,n-1}^{\text{ }k}+\cdots
+a_{k,n-k+i}^{\text{ }k}=\sum\limits_{m=1}^{k-i+1}a_{k,n-m+1}^{k}.
\end{equation*}
\end{thrm}

\begin{proof}
\bigskip From Lemma(1.1) $a_{k,n}^{\text{ }i}-a_{k,n}^{\text{ }i+1}=$ $%
a_{k,n-k+i}^{\text{ }k}.$ Using this equation we write

\begin{eqnarray*}
a_{k,n}^{\text{ }i}-a_{k,n}^{\text{ }i+1} &=&a_{k,n-k+i}^{\text{ }k} \\
a_{k,n}^{\text{ }i+1}-a_{k,n}^{\text{ }i+2} &=&a_{k,n-k+i+1}^{\text{ }k} \\
&&\vdots \\
a_{k,n}^{\text{ }k-1}-a_{k,n}^{\text{ }k} &=&a_{k,n-1}^{\text{ }k}
\end{eqnarray*}%
and adding these equations side by side we obtain%
\begin{equation*}
a_{k,n}^{\text{ }i}-a_{k,n}^{\text{ }k}=a_{k,n-1}^{\text{ }k}+\cdots
+a_{k,n-k+i}^{\text{ }k}
\end{equation*}%
and so
\begin{equation*}
a_{k,n}^{\text{ }i}=a_{k,n}^{\text{ }k}+a_{k,n-1}^{\text{ }k}+\cdots
+a_{k,n-k+i}^{\text{ }k}
\end{equation*}%
which completes the proof.
\end{proof}

\bigskip

\begin{cor}
\bigskip Let $f_{k,n}^{\text{ }i}$ and $\ p_{k,n}^{\text{ }i}$ be the $i$-th
sequences of $k$SO$k$F and $k$SO$k$P respectively. Then, for $1\leq
i<k$
\begin{equation*}
f_{k,n}^{\text{ }i}=f_{k,n}^{\text{ }i+1}+\text{
}f_{k,n-k+i}^{\text{ }k}
\end{equation*}%
and%
\begin{equation*}
p_{k,n}^{\text{ }i}=p_{k,n}^{\text{ }i+1}+\text{
}p_{k,n-k+i}^{\text{ }k}.
\end{equation*}
\end{cor}

\begin{proof}
\bigskip Proof is similar to Lemma(1.1)\bigskip
\end{proof}

\bigskip

\begin{cor}
Let $f_{k,n}^{\text{ }i}$ and $\ p_{k,n}^{\text{ }i}$ be the $i$-th
sequences of $k$SO$k$F and $k$SO$k$P respectively then, for $n\geq 1$ and $%
1\leq i\leq k$%
\begin{equation*}
f_{k,n}^{\text{ }i}=f_{k,n}^{\text{ }k}+f_{k,n-1}^{\text{ }k}+\cdots
+f_{k,n-k+i}^{\text{ }k}=\sum\limits_{m=1}^{k-i+1}f_{k,n-m+1}^{k}
\end{equation*}%
and%
\begin{equation*}
p_{k,n}^{\text{ }i}=p_{k,n}^{\text{ }k}+p_{k,n-1}^{\text{ }k}+\cdots
+p_{k,n-k+i}^{\text{ }k}=\sum\limits_{m=1}^{k-i+1}p_{k,n-m+1}^{k}.
\end{equation*}
\end{cor}

\bigskip

\begin{proof}
Proof is similar in Theorem(1.2)\bigskip
\end{proof}

\bigskip

\begin{exam}
Let us obtain $p_{k,n}^{i}$ for $k=5$, $n=10$ and $i=2$ by using
Corollary
(1.4).%
\begin{eqnarray*}
p_{k,n}^{\text{ }i} &=&p_{k,n}^{\text{ }k}+p_{k,n-1}^{\text{
}k}+\cdots
+p_{k,n-k+1}^{\text{ }k} \\
p_{5,10}^{\text{ }2} &=&p_{5,10}^{\text{ }5}+p_{5,9}^{\text{ }5}+p_{5,8}^{%
\text{ }5}+p_{5,7}^{\text{ }5} \\
&=&4116+1578+605+232 \\
&=&6531
\end{eqnarray*}
\end{exam}

\bigskip

Theorem(1.2) and Corollary(1.4) are important, because there are a
lot of studies on $k-th$ sequences of $k$SO$k$F and $k$SO$k$P(which
are called generalized order-$k$ sequences in some papers[3,5,6]).
Our relations allows to translate these studies to $i-th$ sequences.

\bigskip

\subsection{The determinantal representations}

An $n\times n$ matrix $A_{n}=(a_{ij})$ is called lower Hessenberg matrix if $%
a_{ij}=0$ when $j-i>1$ i.e.,%
\begin{equation*}
A_{n}=\left[
\begin{array}{ccccc}
a_{11} & a_{12} & 0 & \cdots & 0 \\
a_{21} & a_{22} & a_{23} & \cdots & 0 \\
a_{31} & a_{32} & a_{33} & \cdots & 0 \\
\vdots & \vdots & \vdots &  & \vdots \\
a_{n-1,1} & a_{n-1,2} & a_{n-1,3} & \cdots & a_{n-1,n} \\
a_{n,1} & a_{n,2} & a_{n,3} & \cdots & a_{n,n}%
\end{array}%
\right]
\end{equation*}

\begin{thrm}
$\bigskip \lbrack 2]$ $A_{n}$ be the $n\times n$ lower Hessenberg
matrix for
all $n\geq 1$ and define $\det (A_{0})=1,$ then,%
\begin{equation*}
\det (A_{1})=a_{11}
\end{equation*}%
and for $n\geq 2$%
\begin{equation}
\det (A_{n})=a_{n,n}\det
(A_{n-1})+\sum\limits_{r=1}^{n-1}((-1)^{n-r}a_{n,r}\prod%
\limits_{j=r}^{n-1}a_{j,j+1}\det (A_{r-1})).
\end{equation}
\end{thrm}

\bigskip

\begin{thrm}
\bigskip Let $k\geq 2$ be an integer$,$ $a_{k,n}^{\text{ }k}$ be the $k$-th
sequences of $k$SO$k$ and $Q_{k,n}=(q_{st})$ $n\times n$ Hessenberg
matrix$,$
where%
\begin{equation*}
q_{st}=\left\{
\begin{array}{l}
i^{\left\vert s-t\right\vert }\text{ \ \ \ \ \ \ \ \ \ \ \ \ \ \ \ \
\ \ if
\ }-1\leq s-t<k\text{ and }s\neq t, \\
\lambda \text{ \ \ \ \ \ \ \ \ \ \ \ \ \ \ \ \ \ \ \ \ \ \ if\ \ \ \
\ \ \ \ \ \ \ }s=t\text{\ \ ,\ \ \ \ \ \ \ \ \ \ \ \ \ \ \ \ \ \ \ \
\ \ \ \ \ \ \ }
\\
0\text{ \ \ \ \ \ \ \ \ \ \ \ \ \ \ \ \ \ \ \ \ \ \ otherwise\ \ \ \
\ \ \ \
\ \ \ \ \ \ \ \ \ \ \ \ \ \ \ \ \ \ \ \ \ \ \ \ \ \ }%
\end{array}%
\right.
\end{equation*}%
i.e.,%
\begin{equation}
Q_{k,n}=\left[
\begin{array}{cccccc}
\lambda & i & 0 & 0 & \cdots & 0 \\
i & \lambda & i & 0 & \cdots & 0 \\
i^{2} & i & \lambda & i & \cdots & 0 \\
\vdots & \vdots & \vdots & \vdots &  & \vdots \\
i^{k-1} & i^{k-2} & i^{k-3} & i^{k-4} & \cdots & 0 \\
0 & i^{k-1} & i^{k-2} & i^{k-3} & \cdots & 0 \\
& \vdots & \vdots & \vdots & \ddots &  \\
0 & 0 & 0 & \cdots & i & \lambda%
\end{array}%
\right]
\end{equation}%
then%
\begin{equation*}
\det (Q_{k,n})=a_{k,n+1}^{\text{ }k}
\end{equation*}%
where $i=\sqrt{-1}.$
\end{thrm}

\begin{proof}
\bigskip Proof is by mathematical induction on $n$. The result is true for $%
n=1$ by hypothesis.

Assume that it is true for all positive integers less than or equal
to $m,$ namely $\det (Q_{k,m})=a_{k,m+1}^{\text{ }k}.$ Using Theorem
(1.6) we have
\begin{eqnarray*}
\det (Q_{k,m+1}) &=&q_{m+1,m+1}\det
(Q_{k,m})+\sum\limits_{r=1}^{m}((-1)^{m+1-r}q_{m+1,r}\prod%
\limits_{j=r}^{m}q_{j,j+1}\det (Q_{k,r-1})) \\
&=&\lambda \det
(Q_{k,m})+\sum\limits_{r=1}^{m-k+1}((-1)^{m+1-r}q_{m+1,r}\prod%
\limits_{j=r}^{m}q_{j,j+1}\det (Q_{k,r-1})) \\
&&+\sum\limits_{r=m-k+2}^{m}((-1)^{m+1-r}q_{m+1,r}\prod%
\limits_{j=r}^{m}q_{j,j+1}\det (Q_{k,r-1})) \\
&=&\lambda \det
(Q_{k,m})+\sum\limits_{r=m-k+2}^{m}((-1)^{m+1-r}i^{\left\vert
m+1-r\right\vert }\prod\limits_{j=r}^{m}i^{\left\vert
j-j-1\right\vert }\det
(Q_{k,r-1})) \\
&=&\lambda \det
(Q_{k,m})+\sum\limits_{r=m-k+2}^{m}((-1)^{m+1-r}i^{m+1-r}\prod%
\limits_{j=r}^{m}i\det (Q_{k,r-1})) \\
&=&\lambda \det
(Q_{k,m})+\sum\limits_{r=m-k+2}^{m}((-1)^{m+1-r}i^{m+1-r}i^{m+1-r}\det
(Q_{k,r-1})) \\
&=&\lambda \det (Q_{k,m})+\sum\limits_{r=m-k+2}^{m}\det (Q_{k,r-1}) \\
&=&\lambda \det (Q_{k,m})+\det (Q_{k,m-1})+\cdots +\det
(Q_{k,m-(k-1)})
\end{eqnarray*}%
From the hypothesis and the definition of $k-th$ sequences of
$k$SO$k$ we
obtain%
\begin{equation*}
\det (Q_{k,m+1})=\lambda a_{k,m+1}^{\text{ }k}+a_{k,m}^{\text{
}k}+\cdots +\ a_{k,m-k+2}^{\text{ }k}=a_{k,m+2}^{\text{ }k}.
\end{equation*}%
Therefore, the result is true for all non-negative integers.
\end{proof}

\bigskip

\begin{thrm}
\bigskip \bigskip Let $k\geq 2$ be an integer, $a_{k,n}^{\text{ }i}$ be the $%
i$-th sequences of $k$SO$k$ and $Q_{k,n+1}^{i}$ be Hessenberg matrix as;%
\begin{equation*}
Q_{k,n+1}^{i}=\left[
\begin{array}{ccccc}
1 & i & 0 & \cdots & 0 \\
i &  &  &  &  \\
i^{2} &  &  &  &  \\
\vdots &  &  &  &  \\
i^{k-i} &  & Q_{k,n} &  &  \\
0 &  &  &  &  \\
\vdots &  &  &  &  \\
0 &  &  &  &
\end{array}%
\right] _{(n+1)\times (n+1)}
\end{equation*}%
where $Q_{k,n}$ is as (7), then for $2\leq i\leq k$
\begin{equation*}
\det (Q_{k,n}^{i})=a_{k,n}^{\text{ }i}.
\end{equation*}
\end{thrm}

\begin{proof}
Proof \ is similar to the proof of Theorem (1.7).
\end{proof}

\bigskip

\begin{thrm}
\bigskip Let $k\geq 2$ be an integer$,$ $a_{k,n}^{\text{ }k}$ be the $k$-th
sequences of $k$SO$k$ and $B_{k,n}=(b_{ij})$ be an $n\times n$ lower
Hessenberg matrix such that%
\begin{equation*}
b_{ij}=\left\{
\begin{array}{l}
-1\text{ \ \ \ \ \ \ \ \ \ \ \ if \ \ \ }j=i+1, \\
1\text{\ \ \ \ \ \ \ \ \ \ \ \ \ \ if\ \ \ \ \ }0\leq i-j<k\text{
and }i\neq
j \\
\lambda \text{\ \ \ \ \ \ \ \ \ \ \ \ \ \ if \ \ \ \ \ \ \ \ \ \ \ \
\ \ \ \
\ \ \ \ \ }i=j \\
0\text{\ \ \ \ \ \ \ \ \ \ \ \ \ \ otherwise}%
\end{array}%
\right.
\end{equation*}%
i.e.,%
\begin{equation}
B_{k,n}=\left[
\begin{array}{cccccc}
\lambda & -1 & 0 & 0 & \cdots & 0 \\
1 & \lambda & -1 & 0 & \cdots & 0 \\
1 & 1 & \lambda & -1 & \cdots & 0 \\
\vdots & \vdots & \vdots & \vdots &  & \vdots \\
1 & 1 & \cdots & 1 & \cdots & 0 \\
0 & 1 & \cdots & 1 & \cdots & 0 \\
\vdots & \vdots & \vdots & \vdots & \ddots & \vdots \\
0 & 0 & \cdots & 1 & \cdots & \lambda%
\end{array}%
\right]
\end{equation}%
then%
\begin{equation*}
\det (B_{k,n})=a_{k,n+1}^{\text{ }k}.
\end{equation*}
\end{thrm}

\begin{proof}
\bigskip Proof is by mathematical induction on $n$. The result is true for $%
n=1$ by hypothesis.

Assume that it is true for all positive integers less than or equal
to $m,$ namely $\det (B_{k,m})=a_{k,m+1}^{\text{ }k}.$ Using Theorem
(1.6) we have
\begin{eqnarray*}
\det (B_{m+1,k}) &=&b_{m+1,m+1}\det
(B_{k,m})+\sum\limits_{r=1}^{m}((-1)^{m+1-r}b_{m+1,r}\prod%
\limits_{j=r}^{m}b_{j,j+1}\det (B_{r-1,k})) \\
&=&\lambda \det
(B_{k,m})+\sum\limits_{r=1}^{m-k+1}((-1)^{m+1-r}b_{m+1,r}\prod%
\limits_{j=r}^{m}b_{j,j+1}\det (B_{r-1,k})) \\
&&+\sum\limits_{r=m-k+2}^{m}((-1)^{m+1-r}b_{m+1,r}\prod%
\limits_{j=r}^{m}b_{j,j+1}\det (B_{r-1,k})) \\
&=&\lambda \det
(B_{k,m})+\sum\limits_{r=m-k+2}^{m}((-1)^{m+1-r}\prod\limits_{j=r}^{m}(-1)%
\det (B_{r-1,k})) \\
&=&\lambda \det (B_{k,m})+\det (B_{m-1,k})+\cdots +\det
(B_{m-(k-1),k}).
\end{eqnarray*}%
From the hypothesis and the definition of $k$SO$k$ we obtain%
\begin{equation*}
\det (Q_{m+1,k})=\lambda a_{k,m+1}^{\text{ }k}+a_{k,m}^{\text{
}k}+\cdots +\ a_{k,m-k+2}^{\text{ }k}=a_{k,m+2}^{\text{ }k}.
\end{equation*}%
Therefore, the result is true for all non-negative integers.
\end{proof}

\bigskip

\begin{thrm}
\bigskip\ Let $a_{k,n}^{i}$ be $i$-th sequences of $k$SO$k$ and for $k\geq 2$
and $2\leq i\leq k;$
\begin{equation*}
B_{k,n+1}^{i}=\left[
\begin{array}{ccccc}
1 & -1 & 0 & \cdots & 0 \\
1 &  &  &  &  \\
1 &  &  &  &  \\
\vdots &  &  &  &  \\
1 &  & B_{k,n} &  &  \\
0 &  &  &  &  \\
\vdots &  &  &  &  \\
0 &  &  &  &
\end{array}%
\right]
\end{equation*}%
where $B_{k,n}$ is as (8) and the number of 1's in the first column is $%
k-i+1,$ then%
\begin{equation*}
\det (B_{k,n}^{i})=a_{k,n}^{\text{ }i}.
\end{equation*}
\end{thrm}

\bigskip

\begin{proof}
\bigskip Proof is similar to the proof of Theorem (1.9).
\end{proof}

\begin{cor}
\bigskip \bigskip\ If we rewrite Theorem (1.7) and Theorem (1.9) for $%
\lambda =1,$ we obtain%
\begin{equation*}
\det (Q_{k,n})=f_{k,n+1}^{k\text{ }}
\end{equation*}%
and%
\begin{equation*}
\det (B_{k,n})=f_{k,n+1}^{k\text{ }}
\end{equation*}%
respectively.
\end{cor}

\bigskip

\begin{proof}
We know from [8] that for $\lambda =1$, $\det (Q_{k,n})=f_{k,k+n-1}^{\text{ }%
}$, $\det (B_{k,n})=f_{k,k+n-1}^{\text{ }}$ and since $%
f_{k,n}^{k}=f_{k,k+n-2}$ in (1.3) then,%
\begin{equation*}
\det (Q_{k,n})=f_{k,n+1}^{k\text{ }}\text{ and }\det (B_{k,n})=f_{k,n+1}^{k%
\text{ }}.
\end{equation*}
\end{proof}

\begin{cor}
\bigskip \bigskip If we rewrite Theorem (1.7) and Theorem (1.9) for $\lambda
=2,$ we obtain%
\begin{equation*}
\det (Q_{k,n})=p_{k,n+1}^{k\text{ }}
\end{equation*}%
and%
\begin{equation*}
\det (B_{k,n})=p_{k,n+1}^{k\text{ }}
\end{equation*}%
respectively.
\end{cor}

\bigskip

\begin{proof}
\bigskip Proof is similar to the proof of \ Theorem (1.7) for $\lambda =2.$
\end{proof}

\bigskip

\begin{cor}
\bigskip If we rewrite Theorem (1.8) for $\lambda =1$ and $\lambda =2,$ we
obtain%
\begin{equation*}
\det (Q_{k,n}^{i})=f_{k,n}^{\text{ }i}
\end{equation*}%
and%
\begin{equation*}
\det (Q_{k,n}^{i})=p_{k,n}^{\text{ }i}
\end{equation*}%
respectively.
\end{cor}

\bigskip

\begin{cor}
\bigskip \bigskip If we rewrite Theorem (1.10) for $\lambda =1$ and $\lambda
=2,$ we obtain%
\begin{equation*}
\det (B_{k,n}^{i})=f_{k,n}^{\text{ }i}
\end{equation*}%
and%
\begin{equation*}
\det (B_{k,n}^{i})=p_{k,n}^{\text{ }i}
\end{equation*}%
respectively.
\end{cor}

\bigskip

\subsection{The permanent representations}

Let $A=(a_{i,j})$ be a square matrix of order $n$ over a ring R. The
permanent of $A$ is defined by%
\begin{equation*}
\text{per}(A)=\sum\limits_{\sigma \in
S_{n}}\prod\limits_{i=1}^{n}a_{i,\sigma (i)}
\end{equation*}%
where $S_{n}$ denotes the symmetric group on $n$ letters.

Let $A_{i,j}$ be the matrix obtained from a square matrix
$A=(a_{i,j})$ by deleting the $i$-th row and the $j$-th column. Then
it is also easy to see
that%
\begin{equation*}
\text{per}(A)=\sum\limits_{k=1}^{n}a_{i,k}\text{per}(A_{i,k})=\sum%
\limits_{k=1}^{n}a_{k,j}\text{per}(A_{k,j})
\end{equation*}
for any $i,j.$

\bigskip

\begin{thrm}
$\left[ 8\right] $Let $A_{n}$ be $n\times n$ lower Hessenberg matrix
for all
$n\geq 1$ and define per$(A_{0})=1.$ Then,%
\begin{equation*}
\text{per}(A_{1})=a_{11}
\end{equation*}%
and for $n\geq 2$%
\begin{equation}
\text{per}(A_{n})=a_{n,n}\text{per}(A_{n-1})+\sum\limits_{r=1}^{n-1}(a_{n,r}%
\prod\limits_{j=r}^{n-1}a_{j,j+1}\text{per}(A_{r-1})).
\end{equation}
\end{thrm}

\bigskip

\begin{thrm}
\bigskip \bigskip Let $a_{k,n}^{k}$ be the $k$-th sequences of $k$SO$k,$ $%
k\geq 2$ be an integer and $H_{k,n}=(h_{st})$ be an $n\times n$
Hessenberg
matrix$,$ such that%
\begin{equation*}
h_{st}=\left\{
\begin{array}{l}
i^{s-t}\text{ \ \ \ \ \ \ \ \ if \ }-1\leq s-t<k\text{ and }s\neq
t,\text{\
\ \ \ \ \ } \\
\lambda \text{ \ \ \ \ \ \ \ \ \ \ \ if\ \ \ \ \ \ \ \ \ \ \
}s=t\text{,\ \
\ \ \ \ \ \ \ \ \ \ \ \ \ \ \ \ \ \ \ \ \ \ \ \ \ } \\
0\text{ \ \ \ \ \ \ \ \ \ \ \ otherwise\ \ \ \ \ \ \ \ \ \ \ \ \ \ \
\ \ \ \
\ \ \ \ \ \ \ \ \ \ \ \ \ \ \ }%
\end{array}%
\right.
\end{equation*}%
i.e.,%
\begin{equation}
H_{k,n}=\left[
\begin{array}{cccccc}
\lambda & -i & 0 & 0 & \cdots & 0 \\
i & \lambda & -i & 0 & \cdots & 0 \\
i^{2} & i & \lambda & -i & \cdots & 0 \\
\vdots & \vdots & \vdots & \vdots &  & \vdots \\
i^{k-1} & i^{k-2} & i^{k-3} & i^{k-4} & \cdots & 0 \\
0 & i^{k-1} & i^{k-2} & i^{k-3} & \cdots & 0 \\
& \vdots & \vdots & \vdots & \ddots & -i \\
0 & 0 & 0 & \cdots & i & \lambda%
\end{array}%
\right]
\end{equation}%
then%
\begin{equation*}
\text{per}(H_{k,n})=a_{k,n+1}^{\text{ }k}
\end{equation*}%
where $i=\sqrt{-1}.$
\end{thrm}

\bigskip

\begin{proof}
\bigskip Proof is similar to the proof of \ Theorem (1.7) using Theorem
(1.15)
\end{proof}

\bigskip

\begin{thrm}
\bigskip \bigskip \bigskip Let $a_{k,n}^{k}$ be the $k$-th sequences of $k$SO%
$k$P, $k$ $\geq 2$ be an integer and let $D_{k,n}=(d_{st})$ be an
$n\times n$
Hessenberg matrix such that%
\begin{equation*}
d_{st}=\left\{
\begin{array}{l}
1\text{\ \ \ \ \ \ \ \ \ \ \ \ \ \ if \ }-1\leq s-t<k\text{ and }s\neq t,%
\text{\ \ \ \ } \\
\lambda \text{ \ \ \ \ \ \ \ \ \ \ \ \ \ if\ \ \ \ \ \ \ \ \ \ \ }s=t\text{%
,\ \ \ \ \ \ \ \ \ \ \ \ \ \ \ \ \ \ \ \ \ \ \ \ \ \ \ } \\
0\text{ \ \ \ \ \ \ \ \ \ \ \ \ \ otherwise\ \ \ \ \ \ \ \ \ \ \ \ \
\ \ \ \
\ \ \ \ \ \ \ \ \ \ \ \ \ \ \ \ \ }%
\end{array}%
\right.
\end{equation*}%
i.e.,%
\begin{equation}
D_{k,n}=\left[
\begin{array}{cccccc}
\lambda & 1 & 0 & 0 & \cdots & 0 \\
1 & \lambda & 1 & 0 & \cdots & 0 \\
1 & 1 & \lambda & 1 & \cdots & 0 \\
\vdots & \ddots & \ddots & \ddots &  & \vdots \\
1 & 1 & 1 & 1 & \cdots & 0 \\
0 & 1 & 1 & 1 & \cdots & 0 \\
& \ddots & \ddots & \ddots & \ddots &  \\
0 & 0 & 0 & 1 & \cdots & \lambda%
\end{array}%
\right] .
\end{equation}%
where number of 1's in the first column is $k-i+1,$ then%
\begin{equation*}
\text{per}(D_{k,n})=a_{k,n+1}^{\text{ }k}.
\end{equation*}
\end{thrm}

\begin{proof}
\bigskip Proof of the theorem is similar to the proof of \ Theorem (1.9)
using Theorem (1.15)
\end{proof}

\bigskip

\begin{cor}
If we rewrite Theorem (1.16) and Theorem (1.17) for $\lambda =1,$ we obtain%
\begin{equation*}
\text{per}(H_{k,n})=f_{k,n+1}^{\text{ }k}
\end{equation*}%
and%
\begin{equation*}
\text{per}(D_{k,n})=f_{k,n+1}^{\text{ }k}
\end{equation*}%
respectively.
\end{cor}

\bigskip

\begin{proof}
We know from [8] and [7] for $\lambda =1,$
per$(H_{k,n})=f_{k,k+n-1}^{\text{
}}$ and per$(D_{k,n})=f_{k,k+n-1}^{\text{ }}$ respectively and since $%
f_{k,n}^{k}=f_{k,k+n-2}$ in $(1.3)$ then,%
\begin{equation*}
\text{per}(H_{k,n})=f_{k,n+1}^{\text{ }k}\text{ \ and per}%
(D_{k,n})=f_{k,n+1}^{\text{ }k}.
\end{equation*}
\end{proof}

\bigskip

\begin{cor}
\bigskip If we rewrite Theorem (1.16) and Theorem (1.17) for $\lambda =2,$
we obtain
\begin{equation*}
\text{per}(H_{k,n})=p_{k,n+1}^{\text{ }k}
\end{equation*}%
and%
\begin{equation*}
\text{per}(D_{k,n})=p_{k,n+1}^{\text{ }k}
\end{equation*}%
respectively.
\end{cor}

\begin{thrm}
\bigskip \bigskip\ Let $a_{k,n}^{i}$ be the $i$-th sequences of $k$SO$k$ and
for $2\leq i\leq k$ and $n,k\geq 2;$
\begin{equation*}
H_{k,n+1}^{i}=\left[
\begin{array}{ccccc}
1 & -i & 0 & \cdots & 0 \\
i &  &  &  &  \\
i^{2} &  &  &  &  \\
\vdots &  &  &  &  \\
i^{k-i} &  & H_{k,n} &  &  \\
0 &  &  &  &  \\
\vdots &  &  &  &  \\
0 &  &  &  &
\end{array}%
\right]
\end{equation*}%
where $H_{k,n}$ is as (10), then%
\begin{equation*}
\text{per}(H_{k,n}^{i})=a_{k,n}^{\text{ }i}.
\end{equation*}
\end{thrm}

\bigskip

\begin{proof}
\bigskip Proof is similar to the proof of Theorem (1.16).
\end{proof}

\begin{thrm}
\bigskip Let $a_{k,n}^{i}$ be $i$-th sequences of $k$SO$k$ numbers and for $%
2\leq i\leq k$ and $k\geq 2;$%
\begin{equation*}
D_{k,n+1}^{i}=\left[
\begin{array}{ccccc}
1 & 1 & 0 & \cdots  & 0 \\
1 &  &  &  &  \\
1 &  &  &  &  \\
\vdots  &  &  &  &  \\
1 &  & D_{k,n} &  &  \\
0 &  &  &  &  \\
\vdots  &  &  &  &  \\
0 &  &  &  &
\end{array}%
\right]
\end{equation*}%
where $D_{k,n}$ is as (11) and the numbers of $1$'s in the first
column is
$k-i+1,$ then%
\begin{equation*}
\text{per}(D_{k,n}^{i})=a_{k,n}^{i}.
\end{equation*}
\end{thrm}

\bigskip

\begin{proof}
\bigskip \bigskip Proof is similar to the proof of Theorem (1.17).
\end{proof}

\bigskip

\begin{cor}
\bigskip If we rewrite Theorem (1.20) for $\lambda =1$ and $\lambda =2,$ we
obtain%
\begin{equation*}
\text{per}(H_{k,n}^{i})=f_{k,n}^{\text{ }i}
\end{equation*}%
and%
\begin{equation*}
\text{per}(H_{k,n}^{i})=p_{k,n}^{\text{ }i}
\end{equation*}%
respectively.
\end{cor}

\bigskip

\begin{cor}
If we rewrite Theorem (1.21) for $\lambda =1$ and $\lambda =2,$we obtain%
\begin{equation*}
\text{per}(D_{k,n}^{i})=f_{k,n}^{\text{ }i}
\end{equation*}%
and%
\begin{equation*}
\text{per}(D_{k,n}^{i})=p_{k,n}^{\text{ }i}
\end{equation*}%
respectively.
\end{cor}

\bigskip

\begin{thrm}
Let $p_{k,n}^{k}$ be the $k$-th sequences of $k$SO$k$P, $k,n\geq 3$
integers
and $P_{k,n}=(p_{i,j})$ be an $n\times n$ Hessenberg matrix$,$ such that%
\begin{equation*}
p_{i,j}=\left\{
\begin{array}{l}
\left\{
\begin{array}{l}
i)p_{1,1}=\cdots =p_{n-3,n-3}=p_{n-2,n-1}=2\text{\ \ \ \ \ \ } \\
ii)\text{ }p_{n-3,n}=p_{n-4,n}=\cdots =p_{n-k+1,n}=0 \\
iii)\text{\ \ \ \ \ \ \ }1\text{ \ \ \ \ \ \ \ \ \ \ \ \ \ \ \ \ \ \
\ \ \ \
\ otherwise \ \ \ \ \ \ \ \ \ }%
\end{array}%
\right. \text{\ \ \ if \ }i-1\leq j\leq i+k-1 \\
\text{ } \\
0\text{\ \ \ \ \ \ \ \ \ \ \ \ \ \ \ \ \ \ \ \ \ \ \ \ \ \ \ \ \ \ \
\ \ \ \
\ \ \ \ \ \ \ \ \ \ \ \ \ \ \ \ \ \ \ \ otherwise\ \ \ \ \ \ \ \ }%
\end{array}%
\right.
\end{equation*}%
i.e.,%
\begin{equation*}
P_{k,n}=\left[
\begin{array}{ccccccc}
2 & 1 & \cdots  & 1 & 0 & \cdots  & 0 \\
1 & 2 & 1 & \cdots  & 1 & 0 & 0 \\
0 & 1 & \ddots  & \ddots  & \ddots  & 1 & 0 \\
0 & 0 & \ddots  & 2 & 1 & 1 & 0 \\
0 & 0 & 0 & 1 & 1 & 2 & 1 \\
& \ddots  & \ddots  & \ddots  & 1 & 1 & 1 \\
0 & 0 & 0 & 0 & 0 & 1 & 1%
\end{array}%
\right]
\end{equation*}%
then%
\begin{equation*}
\text{per}(P_{k,n})=p_{k,n}^{\text{ }k}.
\end{equation*}
\end{thrm}

$\bigskip $

\begin{proof}
$\bigskip $\bigskip Proof is similar to the proof of \ Theorem (1.9)
by using Theorem (1.15).
\end{proof}

\subsection{Binet's formula for generalized order-$k$ Pell numbers}

Let $\sum\limits_{n=0}^{\infty }a_{n}z^{n}$ be the power series of
the analytical function $f.$ We assume that
\begin{equation*}
f(z)=\sum\limits_{n=0}^{\infty }a_{n}z^{n}\text{ \ when \ }f(0)\neq
0
\end{equation*}%
then the reciprocal of $f(z)$ can be written in the following form%
\begin{equation*}
g(z)=\frac{1}{f(z)}=\sum\limits_{n=0}^{\infty }(-1)^{n}\det
(A_{n})z^{n},
\end{equation*}%
whose radius of converge is inf$\{\left\vert \lambda \right\vert
:f(\lambda
)=0\}[1]$. It is clear that $A_{n}$ is a lower Hessenberg matrix, i.e.,%
\begin{equation*}
A_{n}=\left[
\begin{array}{ccccc}
a_{1} & a_{0} & 0 & \cdots & 0 \\
a_{2} & a_{1} & a_{0} & \cdots & 0 \\
a_{3} & a_{2} & a_{1} & \cdots & 0 \\
\vdots & \vdots & \vdots & \ddots &  \\
a_{n} & a_{n-1} & a_{n-2} &  & a_{1}%
\end{array}%
\right] _{n\times n}.
\end{equation*}%
Let%
\begin{equation}
p_{k}(z)=1+a_{1}z+\cdots +a_{k}z^{k},
\end{equation}%
then the reciprocal of $p_{k}(z)$ is
\begin{equation*}
\frac{1}{p_{k}(z)}=\sum\limits_{n=0}^{\infty }(-1)^{n}\det
(A_{k,n})z^{n}
\end{equation*}%
where%
\begin{equation*}
A_{k,n}=\left[
\begin{array}{ccccc}
a_{1} & 1 & 0 & \cdots & 0 \\
a_{2} & a_{1} & 1 & \cdots & 0 \\
a_{3} & a_{2} & a_{1} & \cdots & 0 \\
\vdots & \vdots & \vdots & \ddots & \vdots \\
a_{k} & a_{k-1} & a_{k-2} &  & 0 \\
0 & a_{k} & a_{k-1} &  & 0 \\
\vdots &  & \ddots & \ddots &  \\
0 & \cdots & a_{k} & \cdots & a_{1}%
\end{array}%
\right] _{n\times n}
\end{equation*}%
Inselberg [4] showed that%
\begin{equation}
\det (A_{k,n})=\sum\limits_{j=1}^{k}\frac{1}{p_{k}^{^{\prime }}(\lambda _{j})%
}(\frac{-1}{\lambda _{j}})^{n+1}\text{ \ \ }(n\geq k)
\end{equation}%
if\ $p_{k}(z)$ has the distinct zeros $\lambda _{j},$ $j=1,2,\ldots
,k.$ Where $p_{k}^{^{\prime }}(z)$ is the derivative of polynomial
$p_{k}(z)$ in (12)

\bigskip

\begin{lem}
\bigskip The equation $-1+3z-z^{2}-z^{k+1}$ does not have multiple roots for
$k\geq 2.$
\end{lem}

\begin{proof}
\bigskip Let $f(z)=1-2z-\cdots -z^{k}$ and let $h(z)=(z-1)$ $f(z).$ Then $%
h(z)=-1+3z-z^{2}-z^{k+1}.$ So $1$ is a root but not a multiple root of $%
h(z), $ since $k\geq 2$ and $f(1)\neq 0.$ Suppose that $a$ is a
multiple root of $h(z)$. Note that $a\neq 0$ and $a\neq 1.$ Since
$a$ is a multiple root,
\begin{equation*}
h(a)=-1+3a-a^{2}-a^{k+1}=0
\end{equation*}%
\begin{equation}
a(3-a-a^{k})=1
\end{equation}%
and
\begin{equation}
h^{^{\prime }}(a)=3-2a-(k+1)a^{k}=0.
\end{equation}%
We obtain
\begin{equation*}
3ka-(k+3)a^{2}=k+1
\end{equation*}%
using equations $(14)$ and $(15)$. Thus $a_{1,2}=\frac{3k\pm \sqrt{%
5k^{2}-16k-12}}{2(k+3)}$ and hence, for $a_{1}$ we get%
\begin{eqnarray*}
0 &=&(k+1)a_{1}^{k}+2a_{1}-3 \\
&=&(k+1)(\frac{3k+\sqrt{5k^{2}-16k-12}}{2(k+3)})^{k}+2(\frac{3k+\sqrt{%
5k^{2}-16k-12}}{2(k+3)})-3
\end{eqnarray*}%
We let
\begin{equation*}
a_{k}=(k+1)(\frac{3k+\sqrt{5k^{2}-16k-12}}{2(k+3)})^{k}+2(\frac{3k+\sqrt{%
5k^{2}-16k-12}}{2(k+3)}).
\end{equation*}%
Then we write:%
\begin{equation*}
a_{k}-3=0.
\end{equation*}%
Since for $k\geq 4,$ $a_{k}<a_{k+1}$ and $a_{4}=\frac{24\,315}{16\,807}=7,$ $%
a_{k}\neq 3,$ a contradiction. It can be easily shown that the roots
of the equation $-1+3z-z^{2}-z^{k+1}=0$ are distinct for $k=3$ and
$k=4$.
Similarly, for%
\begin{equation*}
a_{2}=\frac{3k-\sqrt{5k^{2}-16k-12}}{2(k+3)}
\end{equation*}%
similar results are obtained.

Therefore, the equation $h(z)=0$ does not have multiple roots.

Consequently, from Lemma (1.25) it is seen that the equation
\begin{equation*}
0=1-2z-\cdots -z^{k}
\end{equation*}%
does not have multiple roots for $k\geq 2.$
\end{proof}

\bigskip

\begin{thrm}
Let $p_{k,n}^{k}$ be the $k$-th sequences of $k$SO$k$P, then for
$n\geq k\geq 2$
\begin{equation}
p_{k,n+1}^{k}=\sum\limits_{j=1}^{k}\frac{-1}{p_{k}^{^{\prime }}(\lambda _{j})%
}(\frac{1}{\lambda _{j}})^{n+1}\text{ \ \ }
\end{equation}%
where $p_{k}^{^{\prime }}(z)$ is the derivative of polynomial
$p_{k}(z)$ in (12).
\end{thrm}

\bigskip

\begin{proof}
\bigskip Substituting $a_{1}=-2$ and $a_{i}=-1$ for $2\leq $\bigskip $i\leq k
$ in polynomial $p_{k}(z)$ (1.12)$,$ then%
\begin{equation*}
p_{k}(z)=1-2z-z^{2}-\cdots -z^{k}.
\end{equation*}%
Substituting $a_{1}=-2$ and $a_{i}=-1$ for $2\leq $\bigskip $i\leq
k$ in matrix $A_{k,n}$ and for $\lambda =2$ in matrix $B_{k,n}(8)$
then,
\begin{equation*}
\det (A_{k,n})=(-1)^{n}\det (B_{k,n})
\end{equation*}%
so%
\begin{equation}
\det (B_{k,n})=\frac{\det (A_{k,n})}{(-1)^{n}}.
\end{equation}%
$\det (B_{k,n})=p_{k,n+1}^{k}$ from Corollary(1.13) and using
equation(17)
then%
\begin{equation*}
p_{k,n+1}^{k}=\frac{\det (A_{k,n})}{(-1)^{n}}.
\end{equation*}%
Lemma (1.25) we have the zeros of $p_{k}(z)$ are simple. Hence from
equation(13) we obtain
\begin{equation*}
p_{k,n+1}^{k}=\sum\limits_{j=1}^{k}\frac{-1}{p_{k}^{^{\prime }}(\lambda _{j})%
}(\frac{1}{\lambda _{j}})^{n+1}\text{ \ \ }(n\geq k).
\end{equation*}
\end{proof}

\begin{cor}
\bigskip Let $p_{k,n}^{i}$ be the $i$-th sequences of $k$SO$k$P then for $%
n\geq k\geq 2$ and $1\leq i\leq k;$
\begin{equation*}
p_{k,n}^{i}=\sum\limits_{m=1}^{k-i+1}\sum\limits_{j=1}^{k}\frac{-1}{%
p_{k}^{^{\prime }}(\lambda _{j})}(\frac{1}{\lambda _{j}})^{n}.
\end{equation*}
\end{cor}

\begin{proof}
\bigskip Proof is direct from Corollary (1.4) and Theorem(1.26).
\end{proof}

\bigskip

\begin{exam}
\bigskip Let us obtain $a_{k,n}^{i}$ for $\lambda =2$ and $k=3,$ for $%
\lambda =3$ and $k=3,$ for $\lambda =3$ and $k=4$ by using (1.5);
\begin{equation*}
\begin{tabular}{|c|c|c|c|}
\hline $n\setminus i$ & $1$ & $2$ & $3$ \\ \hline $-2$ & $0$ & $0$ &
$1$ \\ \hline $-1$ & $0$ & $1$ & $0$ \\ \hline $0$ & $1$ & $0$ & $0$
\\ \hline $1$ & $2$ & $1$ & $1$ \\ \hline $2$ & $5$ & $3$ & $2$ \\
\hline $3$ & $13$ & $\mathbf{7}$ & $5$ \\ \hline $4$ & $33$ & $18$ &
$13$ \\ \hline
$\vdots $ & $\vdots $ & $\vdots $ & $\vdots $%
\end{tabular}%
,\text{ \ \ }%
\begin{tabular}{|c|c|c|c|}
\hline $n\setminus i$ & $1$ & $2$ & $3$ \\ \hline $-2$ & $0$ & $0$ &
$1$ \\ \hline $-1$ & $0$ & $1$ & $0$ \\ \hline $0$ & $1$ & $0$ & $0$
\\ \hline $1$ & $3$ & $1$ & $1$ \\ \hline $2$ & $10$ & $4$ & $3$ \\
\hline $3$ & $34$ & $\mathbf{13}$ & $10$ \\ \hline $4$ & $115$ &
$44$ & $34$ \\ \hline
$\vdots $ & $\vdots $ & $\vdots $ & $\vdots $%
\end{tabular}%
,\text{ \ \ }%
\begin{tabular}{|c|c|c|c|c|}
\hline $n\setminus i$ & $1$ & $2$ & $3$ & $4$ \\ \hline $-3$ & $0$ &
$0$ & $0$ & $1$ \\ \hline $-2$ & $0$ & $0$ & $1$ & $0$ \\ \hline
$-1$ & $0$ & $1$ & $0$ & $0$ \\ \hline $0$ & $1$ & $0$ & $0$ & $0$
\\ \hline $1$ & $3$ & $1$ & $1$ & $1$ \\ \hline $2$ & $10$ & $4$ &
$4$ & $3$ \\ \hline $3$ & $34$ & $14$ & $13$ & $10$ \\ \hline $4$ &
$116$ & $\mathbf{47}$ & $\mathbf{44}$ & $34$ \\ \hline
$\vdots $ & $\vdots $ & $\vdots $ & $\vdots $ & $\vdots $%
\end{tabular}%
.
\end{equation*}
\end{exam}

\bigskip

\begin{exam}
\bigskip Let us obtain $a_{k,n}^{i}$ for $\lambda =2$ and $k=3,$ for $%
\lambda =3$ and $k=3,$ for $\lambda =3$ and $k=4$ by using our some
determinant and permanent representation;
\begin{equation*}
\det Q_{3,3}^{2}=\det \left[
\begin{array}{ccc}
1 & i & 0 \\
i & 2 & i \\
0 & i & 2%
\end{array}%
\right] =7,\text{ per}D_{3,3}^{2}=\text{per}\left[
\begin{array}{ccc}
1 & 1 & 0 \\
1 & 3 & 1 \\
0 & 1 & 3%
\end{array}%
\right] =7
\end{equation*}%
and%
\begin{equation*}
\text{per}H_{4,4}^{2}=\text{per}\left[
\begin{array}{cccc}
1 & -i & 0 & 0 \\
i & 3 & -i & 0 \\
-1 & i & 3 & -i \\
0 & -1 & i & 3%
\end{array}%
\right] =47,\text{ per}H_{4,4}^{3}=\text{per}\left[
\begin{array}{cccc}
1 & -i & 0 & 0 \\
i & 3 & -i & 0 \\
0 & i & 3 & -i \\
0 & -1 & i & 3%
\end{array}%
\right] =44
\end{equation*}
\end{exam}

\bigskip %%------------------- --------------------------------
%--------------------------BIBLIOGRAPHY------------------------

\bigskip

\end{document}